\documentclass[11pt]{article}

\usepackage{authblk}
\usepackage{amsmath,amsthm,amssymb,color}
\usepackage{epsfig,amsfonts,latexsym, hyperref, bbm, mathrsfs}
\usepackage{geometry}
\usepackage{setspace}
\onehalfspace

\usepackage{dsfont}

\hypersetup{
    colorlinks=true,
    linkcolor=red,
    citecolor=blue,
    urlcolor=blue,
    pdfborder={0 0 0}
}

\numberwithin{equation}{section}

\newcommand{\pointl}{\mathcal{L}}
\newcommand{\pointq}{\mathcal{Q}}
\newcommand{\scrm}{\mathscr{M}}

\newcommand{\inv}{{-1}}
\newcommand{\bbs}{\mathbb{S}}

\newcommand{\calc}{\mathcal{C}}
\newcommand{\bbn}{\mathbb{N}}

\newcommand{\real}{\mathbb{R}}
\newcommand{\uv}{{\boldsymbol{v}}}
\newcommand{\uu}{{\boldsymbol{\varpi}}}
\newcommand{\Xu}{\boldsymbol{X}_{\uu}}
\newcommand{\Sv}{S_{\uv}}
\newcommand{\Iv}{I_{\uv}}
\newcommand{\nullm}{\emptyset}
\newcommand{\rnz}{\real_0^{\mathbb{N}}}
\newcommand{\scrl}{\mathscr{L}}

\newcommand{\bbt}{\mathbb{T}}
\newcommand{\bfx}{\boldsymbol{X}}

\newcommand{\bbm}{\mathbb{M}}
\newcommand{\stas}{\mbox{St$\alpha$S}}

\newcommand{\beq}{\begin{equation}}
\newcommand{\eeq}{\end{equation}}
\newcommand{\alns}[1]{\begin{align*}#1\end{align*}}
\newcommand{\aln}[1]{\begin{align} #1 \end{align}}
\newcommand{\been}{\begin{enumerate}}
\newcommand{\een}{\end{enumerate}}

\newcommand{\eqd}{\,{\buildrel d \over =}\,}
\newcommand{\ckr}{C_c^+ (\bar{\real}_0)}

\newcommand{\zrer}{\bar{\real}_0}

\newcommand{\point}{\mathcal{P}}
\newcommand{\malpha}{m_\alpha}

\newcommand{\bbo}{\mathbbm{1}}

\newcommand{\tnk}{\tilde{N}_n^{(K)}}

\newcommand{\scrmrz}{\scrm_0}
\newcommand{\aub}{A_\uu^{(B)}}

\newcommand{\calf}{\mathcal{F}}
\newcommand{\dnkb}{N_n'^{ (K,B)}}
\newcommand{\poi}{\mathscr{P}}

\DeclareMathOperator{\distvg}{d_{vg}}
\DeclareMathOperator{\dtvx}{dx}

\newcommand{\upp}{{(p)}}
\newcommand{\upQ}{{(Q)}}
\newcommand{\upq}{{(q)}}
\newcommand{\bbf}{\mathbb{F}}
\newcommand{\tnkb}{\tilde{N}_n^{(K,B)}}
\newcommand{\bfe}{\boldsymbol{e}}
\newcommand{\tpjq}{\bbt_{pj}^\upQ}

\DeclareMathOperator{\dtv}{d}

\DeclareMathOperator{\regvar}{RV}
\DeclareMathOperator{\lamiid}{\lambda_{iid}}

\usepackage{cleveref}
\newtheorem{thm}{Theorem}[section]
\newtheorem{propn}[thm]{Proposition}
\newtheorem{lemma}[thm]{Lemma}

\newtheorem{cor}[thm]{Corollary}

\theoremstyle{remark}
\newtheorem{remark}[thm]{Remark}

\newtheorem{example}{Example}[section]

\theoremstyle{definition}
\newtheorem{defn}[thm]{Definition}

\newtheorem{ass}[thm]{Assumptions}

\DeclareMathOperator{\prob}{\mathbf{P}}
\DeclareMathOperator{\exptn}{\mathbf{E}}
\DeclareMathOperator{\probt}{\prob_{\bbt}}

\DeclareMathOperator{\mbfs}{\mathbf{S}}

\newcommand{\hlconv}{\stackrel{\mbox{\tiny{HL}}}{\longrightarrow}}

\newcommand{\vconv}{\stackrel{v}{\longrightarrow}}

\title{Extremes of Multi-type Branching Random Walks: \\ Heaviest Tail Wins}

\date{}
\allowdisplaybreaks

\author[1]{Ayan Bhattacharya}
\author[1]{Krishanu Maulik}
\author[2]{Zbigniew Palmowski}
\author[1]{Parthanil Roy}
\affil[1]{Statistics and Mathematics Unit, Indian Statistical Institute Kolkata}
\affil[2]{Faculty of Pure and Applied Mathematics, Wroc\l aw University of Science and Technology}

\begin{document}

\maketitle

\begin{abstract}

We consider a branching random walk on a multi($Q$)-type, supercritical Galton-Watson tree which satisfies Kesten-Stigum condition. We assume that the displacements associated with the particles of type $Q$ have regularly varying tails of index $\alpha$, while the other types of particles have lighter tails than that of particles of type $Q$. In this article, we derive the weak limit of the sequence of point processes associated with the positions of the particles in the $n^{th}$ generation. We verify that the limiting point process is a randomly scaled scale-decorated Poisson point process (SScDPPP) using the tools developed in \cite{bhattacharya:hazra:roy:2016}. As a consequence, we shall obtain the asymptotic distribution of the position of the rightmost particle in the $n^{th}$ generation.

\end{abstract}



\section{Introduction} \label{sec:introduction}

The behaviour of the maximal/minimal displacement of a supercritical branching random walk has been a subject of intense studies for a long time. The branching random walk itself appears in lots of applications in a very natural way. Initially, the attention focussed on the light-tailed case of weights where the particles reached its location by a sequence of small jumps and the asymptotic behaviour of the rightmost or the leftmost particle was obtained; see \cite{biggins1976, bramson:1978, hammersley1974, kingman1975}. Later, the second order approximations to the minimal displacement were studied by Hu and Shi \cite{hu:shi:2009} and Addario-Berry and Reed \cite{addario2009}. Similar analyses were made for a branching Brownian motion, see, e.g., \cite{bramson1983}, and for centered minima, see \cite{aidekon2013, BZ2012}. Recently, connections with tree indexed random walk and Gaussian free field were considered; see \cite{bramson:ding:zeitouni:2013} and references therein.

In this article, we focus on displacements with heavy-tailed distribution. The first seminal papers related in this direction are the works of Durrett \cite{durrett:1983, durrett:1979} who proved that the position of the rightmost particle is determined by the largest weight. Other related works on branching random walks with heavy-tailed displacements and their continuous parameter analogues include \cite{kyprianou:1999, gantert2000, lalley:shao:2013, berard2014, maillard:2015}. The point processes associated with branching random walk and branching brownian motion are extensively studied in recent years, see, e.g., \cite{arguin2012poissonian, arguin2013extremal, arguin2011genealogy, aidekon2013branching, madaule:2011, bhattacharya:hazra:roy:2014, bhattacharya:hazra:roy:2016}. This approach provides more insight allowing identification of all order statistics of the position of the particle.

With motivations from \cite{biggins:2005} and the works mentioned above, we analyze the point process induced by a branching random walk with heavy-tailed displacements continuing the research initiated by Bhattacharya et al.\ \cite{bhattacharya:hazra:roy:2016}, and generalize it to the multi-type case, allowing dependence between weights. More formally, in this paper we consider a multi-type branching random walk with heavy tailed increments. We start with branching process with $Q$ types of particles satisfying generalized Kesten-Stigum condition: each particle of type $q\in \{1,\ldots, Q\}$ produces $k_p$ children of type $p$, for all $1 \le p \le Q$, in the next generation with probability $p^\upq(k_1, k_2, \ldots, k_Q)$. Next we attach a random vector (of random length) $\bfx^\upp = (X^\upp_1,\ldots, X^\upp_{k_p})$ to these $k_p$ children, called displacements or weights. We shall assume that in some sense the vector $\bfx^{(Q)}$ is a regularly varying and dominates all other type of vectors $\bfx^{(p)}$ ($p=1,\ldots, Q-1$); see Assumptions~\ref{chap4_ass:displacement_multi} for details.

The position of the particle associated with the vertex $\uv \in \mathbb{V}$ of the corresponding Galton-Watson
tree $\mathbb{V}$ is given by: $\Sv = \sum_{\uu \in \Iv} \Xu$, that is, it is given by its own displacement and sum of displacements of its parents (counted up to the root), where $\Iv$ denotes the path form the root to the vertex $\uv$ and $\Xu$ denotes the collection of displacements $\{\bfx^\upp\}$ corresponding to all types $p$ of offsprings of the vertex $\uv$. In this article, we prove that the weak limit of the point process of appropriately normalized displacements of $n^{th}$ generation:
\aln{
N_n = \sum_{|\uv| =n} \delta_{b_n^{\inv} \Sv} \nonumber
}
(where the scales $b_n$ are defined in Section~\ref{chap4_sec:framework}). It converges to a randomly scaled scale-decorated Poisson point process; see \cite{davydov:molchanov:zuyev:2008, madaule:2011, subag:zeitouni:2014} for other applications of this point process.

As an important corollary, we derive the weak convergence of the rightmost particle in the branching random walk. The main step of the proof is based on showing that the ``heaviest'' displacement dominates the whole behaviour of our branching random walk. Then we follow the core of the idea used in Bhattacharya et al.\ \cite{bhattacharya:hazra:roy:2016}, but here we carry out their arguments with larger complexity. The proof of the main result is broken up into four basic steps. The first basic step is to show that single big jumps appearing in the particle of type $Q$ alone dominate the limiting behaviour of the point process. In the second step we cut the tree at $(n-K)^{th}$ generation. Later, in the third step, we do pruning and regularization to produce fixed $B$ many descendants of each particle and we prove that such modified point process based on this forest is close to the original point process as the truncation level $K$ and pruning range $B$ increase. In the last step we use Laplace functional convergence to prove the main result, which also allows us to identify the limiting point process.

The paper is organized as follows. Next Section introduces main definitions and notations. In Section~\ref{chap4_sec:framework} we present the model with detailed assumptions and the main results. Section~\ref{chap4_sec:sketch} outlines the proof of the main theorem, by breaking down the proof into several lemmas, which are then proved in Section \ref{chap4_sec:proofs} along with the proof of the main result.

\section{Background}\label{sec:notations}

In the following subsections, we shall recall certain concepts discussed in the remaining of the paper. In particular, the notion of regular variation on a Polish space, definition of strictly $\alpha$-stable ($\stas$) point process and basics of multi-type branching random walks.

\subsection{Regular Variation and $\stas$ Point Processes} \label{subsec:regvar_point}
We begin with defining Hult-Lindskog convergence of measures on a Polish space with a zero element. We also recall the definition of regularly varying measures on a Polish space.

Let $(\bbs,d)$ be a Polish space with a continuous scalar multiplication $\bullet: (0, \infty) \times \bbs \to \bbs$ satisfying $1 \bullet s = s$ for all $s \in \bbs$ and $b_1 \bullet (b_2 \bullet s) = (b_1 b_2) \bullet s$ for all $b_1, b_2 > 0$ and for all $s \in \bbs$. We also assume the existence of a zero element $s_0 \in \bbs$ such that  $ b \bullet s_0 = s_0 $ for all $b > 0$. Equip the punctured space $\bbs_0 = \bbs \setminus \{s_0\}$ with the relative topology coming from $\bbs$. We denote by $\bbm(\bbs_0)$ the space of all Borel measures on $\bbs_0$ which are finite outside every neighbourhood of $s_0$, and by $\calc_0$ the collection of all bounded continuous functions $f :\bbs_0 \to [0, \infty)$ that vanishes on some neighbourhood of $s_0$.
\begin{defn}[Hult-Lindskog convergence] A sequence of measures $\{\nu_n (\cdot) : n \ge 1\} \subset \bbm(\bbs_0)$ is said to converge in the Hult-Lindskog sense (see \cite{hult:lindskog:2006}) to a measure $\nu(\cdot) \in \bbm(\bbs_0)$ if
\alns{
\lim_{n \to \infty} \int_{\bbs_0} f(x) \nu_n(\dtvx) = \int_{\bbs_0} f(x) \nu(\dtvx)
}
for all $f \in \calc_0$. This convergence will be denoted by $\nu_n \hlconv \nu$.
\end{defn}

Using this notion of Hult-Lindskog convergence, we next define the regularly varying measures.
\begin{defn}[Regularly varying measure] A measure $\nu \in \bbm(\bbs_0)$ is called regularly varying if there exist $\alpha > 0$, an increasing sequence of positive scalars $\{c_n\}$ satisfying $c_{k n}/ c_n \to  k^{1/\alpha}$ for every $k \in \mathbb{N}$, and a nonzero measure $\lambda \in \bbm(\bbs_0)$ such that
\alns{
n \nu (c_n \bullet \cdot) \hlconv \lambda(\cdot)
}
as $n \to \infty$. We shall denote this by $\nu \in \regvar(\bbs_0, \alpha, \lambda)$.
\end{defn}
It is not difficult to check that the limit measure $\lambda$ satisfies the following scaling property: $\lambda( b \bullet \cdot) = b^{-\alpha} \lambda(\cdot)$ for all $b>0$.

\begin{remark}
  Hult-Lindskog convergence defined above have been named variously in the literature. We provide some instances below. Hult and Lindskog in their original work \cite{hult:lindskog:2006} called it convergence in $\boldsymbol M_0$. In \cite{das:mitra:resnick:2006}, the convergence was extended to what they called as $\mathbb M^* (\mathbb C, \mathbb O)$ convergence. Roughly speaking, they removed a close cone instead of a single point $s_0$. In a more recent work, Lindskog et al.\ \cite{lindskog:resnick:roy:2014} extended this convergence by removing a close set, rather than simply a closed cone. They called this convergence in $\mathbb{M}_{\mathbb{O}}$ or simply $\mathbb{M}_{\mathbb{O}}$-convergence. In view of the originators of the notion, we shall name the convergence after them.
\end{remark}

To illustrate the notion of regularly varying measures, we present two important examples, both of which will play important roles in this paper.
\begin{example}
In the first example, we take $\bbs = \real^\bbn = \{ \boldsymbol{x}=(x_1, x_2, \ldots) : x_i \in \real, i \ge 1\}$ and $s_0 =\boldsymbol 0_\infty := (0, 0, \ldots)$. Let $\boldsymbol{X} = (X_1, X_2, \ldots )$ be the i.i.d. process with
$\prob(X_1 \in \cdot ) \in \regvar(\zrer = [-\infty, \infty] \setminus \{0\}, \alpha, \nu_\alpha(\cdot) )$ where
\aln{
\nu_\alpha (\dtvx) =  \beta \alpha x^{-\alpha -1} \bbo(x>0) \dtvx + (1- \beta) \alpha  (-x)^{-\alpha -1} \bbo(x<0) \dtvx \label{chap4_eq:nu_alpha_explicit}
}
with $0 \le \beta \le 1$ being the tail balancing constant. If $\beta=1$, then $\nu_\alpha$ is denoted by $\malpha(\cdot)$. It was shown in \cite{lindskog:resnick:roy:2014} that $\prob (\boldsymbol{X} \in \cdot) \in \regvar(\rnz= \real^\bbn \setminus \{\boldsymbol{0}_\infty\}, \alpha, \lamiid (\cdot))$ where $\lamiid(\cdot)$ is a measure on $\rnz$ (concentrated on axes) such that
\aln{
\lamiid (\dtvx) = \sum_{i=1}^\infty \bigotimes_{j=1}^{i-1} \delta_0(\dtvx_j) \bigotimes \nu_\alpha(\dtvx_i) \bigotimes_{j=i+1}^\infty \delta_0(\dtvx_j). \label{chap4_eq:defn_lamiid}
}
Here $\delta_0$ denotes the Dirac measure putting unit mass at $0$. More generally, one can obtain other limit measures $\lambda$ (that are not necessarily concentrated on the axes) by introducing dependence among $X_1, X_2, \ldots $; see, for example, \cite{resnick:roy:2014}.
\end{example}

\begin{example}
For the second example, consider the Polish space $\bbs = \scrm(\zrer)$ of all Radon point measures on $\zrer = [-\infty, \infty] \setminus \{0\}$ endowed with vague topology (corresponding vague metric is denoted by $\distvg$). Take $s_0 = \nullm$  to be the null measure. The scalar multiplication by $b>0$ is denoted by $\mbfs_b$ and is defined as follows: if $\point = \sum_i \delta_{u_i} \in \scrm(\zrer)$ then
\alns{
\mbfs_b \point = \sum_i \delta_{b u_i}.
}
The HL convergence in $\scrmrz = \scrm(\zrer) \setminus \{ \nullm \}$ is discussed in and used by \cite{hult:samorodnitsky:2010, fasen:roy:2016} in the context of large deviation for point processes and in \cite{bhattacharya:hazra:roy:2016} in the context of branching random walk.
\end{example}

A point process on $\zrer$ is an $\scrm(\zrer)$-valued random variable defined on $(\Omega, \calf, \prob)$ that does not charge any mass to $\pm \infty$. The definition of strict stability of such point processes is introduced by \cite{davydov:molchanov:zuyev:2008} and has been shown to be connected to regular variation on $\scrmrz$ in \cite{bhattacharya:hazra:roy:2016}.
\begin{defn}[Strictly $\boldsymbol{\alpha}-stable$ point process] A point process $N$ on $\zrer$ (that does not charge any mass to $\pm \infty$) is called a strictly $\alpha$-stable ($\stas$) point process ($\alpha>0$) if for two independent copies $N_1$ and $N_2$ of $N$ and for all $b_1, b_2 > 0$ satisfying $b_1^\alpha + b_2^\alpha =1$, the superposition of the point processes $\mbfs_{b_1} N_1$ and $\mbfs_{b_2} N_2$ has the same distribution as that of $N$. 
\end{defn}

We quote the following result from \cite{bhattacharya:hazra:roy:2016}, which gives a sufficient condition for a point process to be in the superposition domain of attraction of a $\stas$ point process.

\begin{propn}[Bhattacharya et al.\ \cite{bhattacharya:hazra:roy:2016}, Theorem~2.3] \label{fact:superposition_clt}
Let $\pointl$ be a point process on $\zrer$ and $\pointl_i$'s are independent copies of $\pointl$. Suppose that there exists $\alpha >0$ and a non-null measure $m^*(\cdot) \in \bbm(\scrmrz)$ such that $\prob(\pointl \in \cdot) \in \regvar(\scrmrz, \alpha, m^*)$ with $\{c_n : n \ge 1\}$ as the scaling . Then $\pointl$ is in the superposition domain of attraction of a $\stas$ point process $\pointq$, i.e.  $\mbfs_{c_n^\inv} \sum_{i=1}^n \pointl_i \Rightarrow \pointq$. Furthermore in the above situation, Laplace functional of the limiting point process is given by
\alns{
\Psi_\pointq(f) = \exptn \Big( \exp \Big\{ - \pointq(f)\Big\}\Big)= \exp \bigg\{ - \int_{\scrmrz} \bigg( 1- \exp \Big\{ - \int f \dtv \nu \Big\} \bigg) m^*(\dtv \nu) \bigg\}
}
for all non-negative real-valued measurable function $f$.
\end{propn}

For the proof of this Proposition, we refer to \cite{bhattacharya:hazra:roy:2016}.  We shall exploit this Proposition for computing weak limit for a sequence of point processes.

It is established in Example 8.6 of \cite{davydov:molchanov:zuyev:2008}  that the $\stas$ point process admits a special kind of representation. A slightly general notation is introduced in \cite{bhattacharya:hazra:roy:2014} in parallel to the notation introduced in \cite{subag:zeitouni:2014}.
\begin{defn}[Randomly scaled scale-decorated Poisson point process] A point process $M$ is called a randomly scaled scale-decorated Poisson point process (SScDPPP) with intensity measure $m(\cdot)$, scale-decoration $\point$ and a positive  random scale $U$ if
\alns{
M \eqd \mbfs_U \sum_{i=1}^\infty \mbfs_{\lambda_i} \point_i,
}
where $U$ is a positive random variable, $\Lambda = \sum_{i=1}^\infty \delta_{\lambda_i}$ is a Poisson random measure on $(0, \infty)$ with intensity measure $m$ and independent of $U$, and $\{\point_i : i \ge 1\}$ is a collection of independent copies of the point process $\point$, which is also independent of $U$ and $\Lambda$.
\end{defn}
Bhattacharya et al.\ \cite{bhattacharya:hazra:roy:2014, bhattacharya:hazra:roy:2016} established that the limiting point process associated to the branching random walk with marginally regularly varying displacements admits SScDPPP representation.

\section{Framework} \label{chap4_sec:framework}

\subsection{Model} \label{chap4_subsec:model}
We shall consider multi-type branching process with $Q$ type of particles. The root is of type $p$ with probability $\pi(p)$ for $1 \le p \le Q$ where $\pmb{\pi}^t =(\pi(1), \ldots, \pi(Q))$ is a probability vector. Each particle of type $q$ produces $k_p$ children of type $p$ (for all $1 \le p \le Q$) with probability $p^{(q)}(k_1, \ldots, k_Q)$ for $k_j \in \bbn$, $1 \le j \le Q$.

Note that $p^\upq$ is supported on $\bbn^Q$, namely, the multi-type Galton-Watson tree does not have any leaf. In particular, the tree survives with probability $1$ and hence we need not condition on the survival of the tree.  Let $Z_1^{(p)}(q)$ denote the number of children of type $q$ produced by a particle of type $p$ for $p,q \in \{1, \ldots, Q\}$. When the root is of type $p$, then the number of particles is written as $\boldsymbol{Z}_1^{(p)} = (Z_1^{(p)}(1), \ldots, Z_1^{(p)}(Q))$.  The independent copies of $\boldsymbol{Z}_1^{(p)}$ will be denoted by $\{\boldsymbol{Z}_{1,i}^{(p)} : i \ge 1\}$. The vector corresponding to $n^{th}$ generation particles is denoted by $\boldsymbol{Z}_n =(Z_n(1),\ldots, Z_n(Q))$ where $Z_n(p)$ denotes the number of particles at $n^{th}$ generation of type $p$. Let $\mu_{p,q} = \exptn (Z_1^{(p)}(q))$ and $M= ((\mu_{p,q}))$. See \cite{athreya:ney:1972} for more details on multi-type branching process.

\begin{ass}[Assumptions on branching mechanism] \label{chap4_ass:ass_branching_multi}
The assumptions on the branching process are stated as follows:
\let\myenumi\theenumi
\let\mylabelenumi\labelenumi
\renewcommand{\theenumi}{B\myenumi}
\renewcommand{\labelenumi}{{\rm (\theenumi)}}
\been
\item We assume that$\mu_{p,q} < \infty$ for all $p,q \in \{1, \ldots, Q\}$. Since the tree does not have any leaf, $\mu_{p,q}\ge 1$ and hence $M$ is a positively regular matrix i.e., there exists $K$ ($K=1$ in our case) such that every element of $M^K$ is positive. Let $\rho > 0$ denote the maximal (in terms of modulus) eigenvalue of $M$ (such $\rho$ exists and is unique by Perron-Frobenius theory, \cite{karlin:taylor:1975}) is called Perron-Frobenius eigenvalue of $M$. Let $\varsigma$ and $\vartheta$ be the left and right eigenvectors corresponding to $\rho$ satisfying $\varsigma^t\boldsymbol{1} =1$ and $\varsigma^t \vartheta =1$ where $\boldsymbol{1} = (1, \ldots, 1)^t$ (existence and uniqueness of such vectors are again guaranteed by Perron-Frobenius theory). Then we have
$$ M^n= \rho^n P  + R^n \mbox{ such that } |r_{ij}^{(n)}|< \rho_0^n $$
for all $i,j = 1, \ldots, Q$ and $1 < \rho_0 < \rho$, where $P = \vartheta. \varsigma^t$ i.e. $P_{ij} = \vartheta_i \varsigma_j$ and $r_{ij}^{(n)}$ denotes the $(i,j)^{th}$ elements of $R^n$; see \cite{karlin:taylor:1975}. Since each $\mu_{pq} \ge 1$, it follows that $\rho \in (1, \infty)$ and hence the underlying multi-type Galton-Watson tree is supercritical. \label{chap4_ass:branching_frobenious}

\item (Kesten-Stigum conditions:) \label{chap4_ass:branching_Kesten_Stgum} For all $p,q \in \{1, \ldots, Q\}$, we assume
 $$\exptn \Big(Z_1^{(p)}(q) \log Z_1^{(p)}(q) \Big) < \infty.$$
\een
\end{ass}

Then from \cite{kesten:stigum:1966a} we get,
\aln{
\lim_{n \to \infty} \frac{1}{\rho^n} \boldsymbol{Z}_n = W \varsigma \label{chap4_eq:Kesten_Stigum_limit_multi}
}
almost surely for some positive random variable $W$.

We are now ready to describe the model formally in terms of the tree given in Assumptions above and the displacements of each of its vertices. We recall the notations for the position of the particles born in the branching process from Section~\ref{sec:introduction}. Let $\bbt = (\mathbb{V},\mathbb{E})$ denote the genealogical tree obtained from the supercritical multi-type Galton-Watson process where $\mathbb{V}$ and $\mathbb{E}$ denote the collection of all vertices and edges (of $\bbt$) respectively. We use the Harris-Ulam labels for the vertices of $\bbt$ as described in \cite{mode:1971}. The root is denoted by $o$. We also identify each edge with its vertex away from the root and assign its type to the edge. The displacements of the particles are described next. We assume that each particle of type $p$ produces an independent copy of $(\scrl^{(p)}(1), \scrl^{(p)}(2), \ldots, \scrl^{(p)}(Q))$, where
\aln{
\scrl^{(p)}(q) = \sum_{i=1}^{Z_1^{(p)}(q)} \delta_{X_i^{(q)}},
}
with $\bfx^{(q)} = (X^{(p)}_1, X^{(p)}_2, \ldots )$ being a random element of $\real^\bbn$ which is independent of the random vector $\boldsymbol{Z}_1^\upp$. This means the following: if a particle has $k_q (\in \bbn)$ children of type $q$ in the next generation (for all $1 \le q \le Q$), then the displacements associated to the children are independent copy of $\Big( X_1^{(1)}, \ldots, X_{k_1}^{(1)}, \ldots, X^\upQ_{1}, \ldots, X^\upQ_{k_Q} \Big)$.

Now we define the position of a particle as the sum of its own displacement and displacements of its parents up to the root $o$. For mathematical description, let $\Iv$ denote the path from the root $o$ to the vertex $\uv \in \mathbb{V}$ and $\Xu$ denote the displacement attached to the $\uu^{th}$ particle (see the convention mentioned above). Then the position of the $\uv^{th}$ particle is denoted by $\Sv$ and defined as
\alns{
\Sv = \sum_{\uu \in \Iv} \Xu.
}

\begin{ass}[Assumptions on Displacements] \label{chap4_ass:displacement_multi}
\let\myenumi\theenumi
\let\mylabelenumi\labelenumi
\renewcommand{\theenumi}{D\myenumi}
\renewcommand{\labelenumi}{{\rm (\theenumi)}}
For each $p \in \{1, \ldots, Q\}$, $ \boldsymbol{X}^\upp = (X_1^\upp, X_2^\upp, \ldots) $ is an $\real^\bbn$-valued random variable  such that the following assumptions are satisfied:
\been

\item For every fixed $p$, $\{X_i^{(p)} : i \ge 1\}$ is a collection of marginally identically distributed random variables. \label{chap4_ass:disp_iid_pp}

\item  There exists an increasing sequence of positive constants $\{b_n : n \ge 1\}$ such that\label{chap4_ass:disp_typeQ}

\been
\item (Condition on marginal distribution)
\aln{
\rho^n \prob(b_n^\inv X_1^{(Q)} \in \cdot) \vconv \nu_\alpha(\cdot), \label{chap4_eq:marginal_bn}
}
where $\nu_\alpha(\cdot)$ is a measure on $\zrer$ as defined in \eqref{chap4_eq:nu_alpha_explicit}.

\item (Condition on joint distribution of the displacements associated to the particles of type $Q$)\
\aln{
\rho^n \prob\Big( b_n^\inv \bfx^{(Q)} \in \cdot \Big)  \hlconv \lambda(\cdot), \label{chap4_eq:hlconv_bn}
}
where $\lambda$ is a measure on $\rnz = \real^\bbn \setminus \{0_\infty\}$, such that $\lambda(a \cdot) = a^{-\alpha} \lambda(\cdot)$ for every $a > 0$, and $\lambda(A) < \infty$ if $0_\infty\notin \bar{A}$.
\een

\item For all $1 \le p \le Q-1$, there exists $\gamma >0$, such that \label{chap4_ass:disp_otherthanQ}
\aln{
\prob(|X^{(p)}_1|>x) \le \varepsilon(x) \prob(|X_1^{(Q)}| > x)
}
such that $\varepsilon : \real_+ \to \real_+ $ is a map satisfying,
\aln{
\lim_{t \to \infty}\dfrac{\varepsilon(tx)}{\varepsilon(t)} = x^{-\gamma}.\label{varepsilon}
}
for all $x>0$.
\een
\end{ass}

Let $|\uv|$ denote the generation of the vertex $\uv \in \mathbb{V}$ in the genealogical tree $\bbt$. In this article, we shall obtain the weak limit of the point process,
\aln{
N_n = \sum_{|\uv| =n} \delta_{b_n^{\inv} \Sv}. \label{chap4_eq:defn_Nn_multi}
}

\subsection{Results}\label{chap4_subsec:mainresults}
We need to introduce the following quantities to describe the limiting point process.

 Let $G$ be a positive integer-valued random variable with probability mass function
\aln{
\prob(G=g) = \sum_{q=1}^Q \varsigma_q \prob(U_1^{(q)}(Q) = g), ~~~~~~~~ g \in \bbn,
}
where $\{U_1^{(q)}, q=1,\ldots,Q\}$ is a collection of independent random variables such that for each $q$, $U_1^\upq \eqd Z_1^\upq(Q)$.

We denote by $(Z^{(p)}_{m} (1), \ldots, Z^{(p)}_{m}(Q))$ the random vector of the number of particles at the $m^{th}$ generation for different types when root is of type $p$ with probability one.
Now consider a positive integer sequence-valued random variable (stochastic process) $\boldsymbol{T} = \{T_i: i \ge 1 \}$ such that for every fixed $i \in \bbn$,
\aln{
\prob \Big((T_1, \ldots, T_i) = (t_1, \ldots, t_i)\Big) =(\rho-1) \sum_{m=0}^\infty \frac{1}{\rho^{m+1}} \prod_{j=1}^i \prob\Big( \sum_{p=1}^Q Z_m^\upQ (p) = t_j \Big).
}
Let
\alns{
\poi = \sum_{l=1}^\infty \delta_{\pmb{\xi}_l} = \sum_{l=1}^\infty \delta_{(\xi_{l1}, \xi_{l2}, \ldots)}
}
be a Poisson random measure (PRM) on $\real^{\bbn} \setminus \{ \boldsymbol{0}\}$ with intensity measure $\lambda(\cdot)$ which is described in \eqref{chap4_eq:hlconv_bn}. We assume that $\poi$ is independent of $W$.

Let $\{G_l : l \ge 1\}$ be i.i.d. copies of the random variable $G$ and is also independent of $W$ and $\poi$.  Consider the collection $\{\boldsymbol{T}_l = (T_{l1}, T_{l2}, T_{l3}, \ldots): l \ge 1\}$ consisting of independent copies of $\boldsymbol{T}$ and assume further that this collection is also independent of  $\{G_l: l \ge 1\}$, $W$ and $\poi$.

\begin{thm} \label{thm:main_thm_multi}
Suppose Assumptions \ref{chap4_ass:ass_branching_multi} and \ref{chap4_ass:displacement_multi} hold. Consider the point process sequence $N_n$ defined by \eqref{chap4_eq:defn_Nn_multi} with $b_n$ described in \eqref{chap4_eq:marginal_bn} and \eqref{chap4_eq:hlconv_bn}. Then $N_n$ converges weakly as $n \to \infty$ to the Cox cluster process
\aln{
N_* = \sum_{l=1}^\infty \sum_{k=1}^{G_l} T_{lk} \delta_{((\rho-1)^\inv W)^{1/ \alpha} \xi_{lk}} \label{chap4_eq:limiting_pp_multi}
}
in the space $\scrm(\rnz)$ equipped with vague topology. Moreover, the limiting point process $N_*$ is a randomly scaled scale-decorated Poisson point process (SScDPPP) as defined in Subsection \ref{subsec:regvar_point}.
\end{thm}

Now we shall discuss the consequences of Theorem \ref{thm:main_thm_multi}. Let $M_n = \max_{|\uv| = n} \Sv$ denote the position of the rightmost particle at the $n^{th}$ generation. Then the following result gives is the multi-type extension of the main result of \cite{durrett:1983}. We need the following notations to write down the asymptotic distribution of $M_n$ after scaling by $b_n$.

Let $G_0 = [-\infty, 1)$ and $G_1 = (1, \infty ]$. Assume that for every fixed $g$, $H_{i_1, \ldots, i_g} = G_{i_1} \times G_{i_2} \times \ldots \times G_{i_g} \times \real \times \ldots \times \real$ for $i_j \in \{0,1\}$ and $1 \le j \le g$. It is important to note that $\{ H_{i_1, \ldots, i_g} : i_j \in \{0,1\} \mbox{ and } 1 \le j \le g\}$ is a disjoint collection of sets.

The proofs of the following corollaries are similar to the proofs in Section~5 in \cite{bhattacharya:hazra:roy:2016}.

\begin{cor} \label{cor:maxima}
Under the assumptions of Theorem \ref{thm:main_thm_multi},
\aln{
\lim_{n \to \infty} \prob(M_n < b_n x) = \exptn \bigg[  \exp \Big\{ - \kappa_\lambda W x^{-\alpha}\Big\}\bigg],
}
where
\begin{equation}\label{chap4_eq:kappa_lambda_multi}
\kappa_\lambda=(\rho-1)^\inv \sum_{g=1}^\infty \prob(G=g) \sum_{\stackrel{i_1, \ldots, i_g \in \{0,1\}^g :}{ 1 \le i_1 + \ldots+ i_g \le g}} \lambda(H_{i_1, \ldots, i_g}).
\end{equation}
\end{cor}

Finally we consider two special cases, where an explicit SScDPPP representation can be obtained; see the corollary and remark below.

\begin{cor} \label{cor:iid}
Under the assumptions of Theorem \ref{thm:main_thm_multi} with $\boldsymbol{X}^\upQ$ as an i.i.d. process, the sequence of point processes converge to $N_{iid} \sim SScDPPP(m_\alpha(\cdot), H \delta_{\chi}, ((\rho-1)^\inv \varsigma_Q W)^{1/\alpha})$ where $\chi$ is a $\{\pm 1\}$-valued random variable with $\prob(\chi=1)=p$ and $H$ is a positive integer-valued random variable with probability mass function
\alns{
\prob(H=y) = (\rho-1) \sum_{m=0}^\infty \frac{1}{\rho^{m+1}} \prob \bigg( \sum_{p=1}^Q Z_m^\upQ(p) =y \bigg).
}
\end{cor}

\begin{remark}
Assume that $Z^\upp_1 (Q) \le B $ for every $1 \le p \le Q$. In this case, following exactly the same lines as in Corollary 5.3
in \cite{bhattacharya:hazra:roy:2016}, it is easy to obtain explicit  SScDPPP representation of the limiting point process even when $X_1^\upQ, X_2^\upQ, \ldots$ are not independent.
\end{remark}

\section{Sketch of Proof of the Theorem \ref{thm:main_thm_multi}}\label{chap4_sec:sketch}

Recall that $\Iv$ denotes the geodesic path from the root $o$ to the vertex $\uv$. Let $\Iv^\upp$ denote the collection of all edges of type $p$ (see the convention mentioned in Section \ref{chap4_sec:framework}) on the path $\Iv$. Let $D_n : = \{\uv \in \mathbb{V}: |\uv| = n\}$ denote the collection of all the $n^{th}$ generation vertices of $\bbt$ and $D_n^\upp$ is the collection of all vertices in $D_n$ of type $p$. By $|D_n|$, we denote the cardinality of the random set $D_n$. At the beginning, note that on a fixed path, the displacements are independent. Another important observation is that  on a fixed path, the contributions of the $p^{th}$ type of particles are negligible with high probability for sufficiently large $n$ for each $p=1, \ldots, Q-1$. The reason is that the displacements corresponding to them has lighter tails than those of the $Q^{th}$ type of particles. Final observation concerns the fact that on a fixed path at most one of the $Q^{th}$ type of particle is large which is a consequence of the well known ``one large jump" principle for sum of i.i.d. random variables with regularly varying tails. Define
\aln{
\tilde{N}_n = \sum_{|\uv|=n} \sum_{\uu \in \Iv^{(Q)}} \delta_{b_n^\inv \Xu}.
}
To formalize all above observations it is enough to show that for sufficiently large $n$
the sequences of point processes $N_n$ and $\tilde{N}_n$ are close in vague topology.
\begin{lemma} \label{chap4_lemma:one_large_jump}
With the notations defined above,
\aln{
\limsup_{n\to \infty} \prob \Big( \distvg(N_n, \tilde{N}_n) > \epsilon \Big) =0 \label{chap4_eq:aim_one_large_jump_multi}
}
for every $\epsilon >0$.
\end{lemma}

 After this step we ignore the displacements associated to the particles of type $p \in \{1, \ldots, Q-1\}$,  but we do not delete the vertices (or the edges) from the tree $\bbt$ to keep the tree structure unchanged.

Then we follow the twofold truncation technique (from \cite{bhattacharya:hazra:roy:2016}) for the multi-type Galton-Watson tree $\bbt$.
For the first one, we fix an integer $K<n$ and we cut the tree at the $(n-K)^{th}$ generation. After cutting the tree $\bbt$, we are left with $|D_{n-K}|$ many independently distributed multi-type Galton-Watson trees consisting of $K$ generations, with the particles in $D_{n-K}$ as the roots. For each $\uv \in D_n$, define $\Iv^\upQ(K) = \{ \uu \in \Iv^\upQ : |\uu \to \uv| \le K \}$ where $\uu \to \uv$ denotes the unique geodesic path from $\uu$ to $\uv$. Define a new point process (associated to the forest obtained after cutting) as
\alns{
\tnk = \sum_{|\uv|=n} \sum_{\uu \in \Iv^\upQ (K)} \delta_{b_n^\inv \Xu}.
}
The following  can easily be derived from Subsection~4.1 in \cite{bhattacharya:hazra:roy:2014} making obvious changes in notations:
\aln{
\lim_{K \to \infty} \limsup_{n \to \infty} \prob \bigg[ \distvg(\tilde{N}_n, \tnk) > \epsilon \bigg] =0
}
for every $\epsilon>0$.

Now we shall prune the forest obtained after cutting. There are $|D_{n-K}^\upp|$ many sub-trees consisting of $K$ generations with roots of type $p$ for $p=1, \ldots, Q$. Let $\bbf_p$ denote the collection of all sub-trees with root of type $p$ for every $p=1, \ldots, Q$. For a fixed $p$, the members of the forest $\bbf_p$ is enumerated as $\bbt_{p1}, \ldots, \bbt_{p |{D_{n-K}}^\upp|}$.  Suppose $\bbt_{pj}^\upQ$ denotes the collection of all vertices of type $Q$ in the tree $\bbt_{pj}$ for every $j = 1, \ldots, |D_{n-K}^\upp|$ and $p=1, \ldots, Q$.
Observe that
\aln{
 \tnk = \sum_{p=1}^Q \sum_{j =1}^ {Z_{n-K}(p)} \sum_{\uu \in \bbt_{pj}^\upQ} A_{\uu} \delta_{b_n^\inv \Xu}, \label{chap4_eq:defn_tnk}
}
where $A_{\uu}$ denotes the total number of descendants of the particle $\uu$ at the $K^{th}$ generation of the sub-tree $\bbt_{pj}$ containing $\uu$.

Define for every $B \in \bbn$,
\alns{
Z_1^\upp (q,B) = Z_1^\upp(q) \bbo(Z_1^\upp(q) \le B) + B \bbo(Z_1^\upp (q) > B)
}
which is the random variable $Z_1^{(p)}(q)$ truncated at $B$ (for every $p,q=1, \ldots, Q$). Let $\mu_{pq}(B) = \exptn (Z_1^\upp(q, B))$ be the $(p,q)^{th}$ element of the $Q \times Q$ matrix $M(B)$ for every $p,q=1, \ldots, Q$. Assume that $\rho_B$ be the Perron-Frobenius eigenvalue of $M(B)$. Since each $\mu_{pq}(B) \ge 1$; using the same argument as in \eqref{chap4_ass:branching_frobenious}. $\rho_B > 1$. Finally we shall prune the sub-trees $\{\bbt_{pj} : 1 \le p \le Q, 1\le j \le Z_{n-K}(p), \}$ according to the following algorithm:
\been
\item[P1] Fix $p=1$ and $j=1$. 

\item[P2] Look at the root of $\bbt_{11}$. If it has more than $B$ descendants of type $1$ in the next generation, keep $B$ of them and discard others. Otherwise, do nothing. Do it for the other types of descendants in the first generation.

\item[P3] Repeat  Step P2 to all the particles at generation $1$. Continue till $K$-th generation.

\item[P4] Repeat Step P2 and P3 for the other trees in $\bbf_1$.

\item[P5] Repeat Step P2, P3 and P4 for the other forests $\bbf_p$ for $2\le p \le Q$.
\een

The pruned $j^{th}$ member of the forest $\bbf_{p}$ is denoted by $\bbt_{pj}(B)$ for all $p=1, \ldots,  Q$ and $j=1, \ldots,  |D_{n-K}^\upp|$. Define
\alns{
\tnkb = \sum_{p=1}^Q \sum_{j=1}^{Z_{n-K}(p)} \sum_{\uu \in \bbt_{pj}^\upQ(B)} \aub \delta_{b_n^\inv \Xu}
}
which is the point process associated to the pruned forest. The next step is to note that
\alns{
\lim_{B \to \infty} \limsup_{n \to \infty} \prob \Big( \distvg(\tnk, \tnkb) > \epsilon \Big) =0
}
for every $\epsilon>0$ and fixed integer $K$ which can be established in parallel to the proof of Lemma~4.2 in \cite{bhattacharya:hazra:roy:2016} and Lemma~3.3 in \cite{bhattacharya:hazra:roy:2014}.

Final step in proving the weak convergence is given in next lemma.

\begin{lemma} \label{lemma:weaklimit_multi}

\let\myenumi\theenumi
\let\mylabelenumi\labelenumi
\renewcommand{\theenumi}{L\myenumi}
\renewcommand{\labelenumi}{{\rm (\theenumi)}}

Under the assumptions of Theorem \ref{thm:main_thm_multi} for all $K \ge 1$ and $B$ large enough so that $\rho_B>1$, there exist point processes $N_*^{(K,B)}$ and $N_*^{(K)}$ such that
\been
\item $\tnkb \Rightarrow N_*^{(K,B)}$ as $n \to \infty$,  \label{chap4_item:weaklim1}
\item $N_*^{(K,B)} \Rightarrow N_*^{(K)}$ as $B \to \infty$, \label{chap4_item:weaklim2}
\item and $N_*^{(K)} \Rightarrow N_*$ as $K \to \infty$ \label{chap4_item:weaklim3}
\een
in the space of all Radon measures on $[-\infty, \infty] \setminus \{0\}$, equipped with vague topology. Furthermore, $N_*$ admits the representation in \eqref{chap4_eq:limiting_pp_multi} and also an SScDPPP representation.
\end{lemma}

To prove this lemma we shall modify and use the technique ``regularization" as used in \cite{bhattacharya:hazra:roy:2016}, described in Subsection \ref{chap4_subsec:weak}. The main observations leading to the proof is as follows. Note that after cutting and pruning the point process $\tnkb$ can be written as the superposition $|D_{n-K}|$ point processes which are independently distributed. Define
\aln{
N'^{(K,B,p,j)} = \sum_{\uu \in \bbt_{pj}^\upQ(B)} \aub \delta_{\Xu}
}
where $\aub$ denotes the number of descendants in the $K^{th}$ generation of the vertex $\boldsymbol{\varpi}$   and $N'^{(K,B,p,j)}_n = \mbfs_{b_n^\inv} N'^{(K,B,p,j)}$. The following lemma shows that, each of these point processes are regularly varying (in the sense of \cite{hult:lindskog:2006}) in the space $\scrmrz$.

\begin{lemma} \label{lemma:hlconv_t11_multi}
$N'^{(K,B)}_n$ is the superposition of $|D_{n-K}|$ many independent point processes $\{N'^{(K,B,p,j)}_n : p = 1, 2 , \ldots, Q \mbox{ and } j = 1,2, \ldots, |D_{n-K}^\upp|\}$ and there exists non-null measures $\{m_p^*: p=1,2, \ldots,Q\}$ such that,
\alns{
m_n^{(p)} = \rho^n \prob \Big( \mbfs_{b_n^\inv} N^{(K,B,p,j)} \in \cdot \Big) \hlconv m_p^*(\cdot)
}
for every $1 \le j \le |D_{n-K}^\upp|$ and $p=1, 2, \ldots, Q$.
\end{lemma}

Finally using superposition limit theorem  described in Proposition~\ref{fact:superposition_clt} (see Theorem~2.3 in \cite{bhattacharya:hazra:roy:2016}), we shall establish \eqref{chap4_item:weaklim1}. The proofs of \eqref{chap4_item:weaklim2} and \eqref{chap4_item:weaklim3} is in parallel to single type case (see \cite{bhattacharya:hazra:roy:2014, bhattacharya:hazra:roy:2016}.)

\section{The Proofs} \label{chap4_sec:proofs}

\subsection{Heaviest Tail Wins}\label{chap4_subsec:heaviest}

We start with the following consequence of Karamata theory.

\begin{lemma}\label{chap4_lemma:upper_bound_varepsilon}
Fix $\gamma' \in (0, \gamma)$ and $\beta \in (0, \alpha^\inv)$. For $\varepsilon$ defined in \eqref{varepsilon}
there exists $N \in \bbn$, such that $n \ge  N$ and $b_n$ as chosen in \eqref{chap4_eq:hlconv_bn} and  \eqref{chap4_eq:marginal_bn},
\aln{
\epsilon (b_n) \le \mbox{ const. }\rho^{-n(\gamma - \gamma')(1/\alpha - \beta)}, \label{chap4_eq:uuper_bound_varepsilon}
}
 and some $0 < \gamma' < \gamma$ and $0 < \beta < 1/ \alpha$.
\end{lemma}

\begin{proof}
For any $t\geq 0$, let $b(t)$ be a function such that
$b(n)=b_n$. Let $c(t) = b(\log_\rho t)$ and $\bar{F}_Q(x) = \prob(|X_1^\upQ| > x)$.
From \eqref{chap4_eq:hlconv_bn} it follows that
\alns{\rho^n \bar{F}_Q(b_n x) \to x^{-\alpha}
}
for all $x > 0$. Now using Proposition 0.1 of \cite{resnick:1987} we get that for all $x>0$,
\aln{
\dfrac{1}{b_n} U (\rho^n x) \to x^{1/\alpha},
}
where $U = \bigg( \dfrac{1}{\bar{F}_Q} \bigg)^{\leftarrow}$.  Observe that
\alns{
\lim_{t \to \infty} \dfrac{b(t+x)}{b(t)} = \lim_{t \to \infty} \dfrac{b(t+x)}{U(\rho^{t+x})} \dfrac{U(\rho^t \rho^x)}{b(t)} = \rho^{x/ \alpha}
}
and that
\alns{
\lim_{t \to \infty} \dfrac{c(tx)}{c(t)} = \lim_{t \to \infty} \dfrac{b(\log_\rho t + \log_\rho x)}{b(\log_\rho t)} = \rho^{\alpha^\inv \log_\rho x} = x ^{1/\alpha}.
}
Hence $c(\cdot) \in \regvar_{1/\alpha}$.

Using Karamata representation (see Theorem 0.6 in  \cite{resnick:1987}) we can write:
\alns{
c(t) = t^{1/\alpha} \phi_1(t) \exp \Big\{ \int_{1}^t \dfrac{1}{x} \eta_1(x ) \dtvx  \Big\},
}
where $ \phi_1 : \real_+ \to \real_+$ and $\eta_1 : \real _+ \to \real_+$ such that
\aln{
\lim_{t \to \infty} \phi_1(t ) = \phi_1 \in (0, \infty) \mbox{ and } \lim_{t \to \infty} \eta_1(t) =0.
}
We also have assumed that $\varepsilon( \cdot) \in \regvar_\gamma$ for some $\gamma > 0$. Thus
\aln{
\varepsilon(t) = t^{-\gamma} \phi_2(t) \exp \Big\{ \int_1^t \dfrac{1}{x} \eta_2(x) \dtvx \Big\},
}
where $\phi_2 : \real_+ \to \real_+$ and $\eta_2 : \real_+ \to \real_+$ such that
\aln{
\lim_{t \to \infty} \phi_2(t) = \phi_2 \in (0, \infty) \mbox{ and } \lim_{t \to \infty} \eta_2(t) =0.
}

We want to obtain an upper bound of
\aln{
\varepsilon(b(n)) &= \varepsilon( b(\log_\rho \rho^n)) \nonumber \\
& = \varepsilon(c(\rho^n)) \nonumber \\
&= \Big(c(\rho^n) \Big)^{-\gamma} \phi_2(c(\rho^n)) \exp \Big\{ \int_1^{(c(\rho^n))} \frac{1}{x} \eta_2(x) \dtvx \Big\}.
}
Now fix $\delta_1 > 0$ and $ \gamma' \in (0, \gamma)$. Then there exits $N_1 \in \bbn$ such that for all $n \ge N_1$,
\aln{
\phi_2(c \rho^n)  \exp \Big\{ \int_1^{(c(\rho^n))} \frac{1}{x} \eta_2(x) \dtvx \Big\} < \mbox{ const. }(\phi_2 + \delta_1) \Big(c(\rho^n) \Big)^{\gamma'}.
}
Now we derive the following inequality for all $n \ge N_1$:
\aln{
\varepsilon(b(n)) < \mbox{ const. } \Big(c(\rho^n) \Big)^{-(\gamma - \gamma')} = \mbox{ const. } \Bigg( \rho^{n/\alpha} \phi_1(\rho^n) \exp \Big\{ \int_1^{\rho^n} \frac{1}{x} \eta_1(x)  \dtvx \Big\} \Bigg)^{-(\gamma -\gamma')}. \label{chap4_eq:upperbound_varepsilon}
}
Fixing $\delta_2 > 0$, $\beta \in (0, \alpha^\inv)$ and using similar arguments we can prove that there exists $N_2 \in \bbn$ such that for all $n \ge N_2$,
\aln{
\phi_1(\rho^n) \exp \Big\{ \int_1^{\rho^n} \frac{1}{x} \eta_1(x) \dtvx \Big\} > (\phi_1+ \delta_2)  \rho^{-n\beta}. \nonumber
}
Combining this inequality with \eqref{chap4_eq:upperbound_varepsilon} proves \eqref{chap4_eq:uuper_bound_varepsilon} with $N= \max(N_1, N_2)$.
\end{proof}

\begin{proof}[Proof of Lemma \ref{chap4_lemma:one_large_jump}]
Consider a Lipschitz function $f \in \ckr$ with $\mbox{\rm support} (f) = \{x : |x| > \zeta \}$, where $\ckr$ denotes the class of all non-negative, bounded and continuous functions. To prove  \eqref{chap4_eq:aim_one_large_jump_multi}, it is enough to show that
\aln{
\limsup_{n \to \infty} \prob \Big( |N_n(f) - \tilde{N}_n(f)| > \epsilon \Big) = 0 \label{chap4_eq:goal_one_large_jump}
}
for every $\epsilon > 0$.

In order to prove this, at the beginning we will formalize the fact that on a fixed path the contribution of the $p^{th}$ type of particles is asymptotically negligible for $p=1, \ldots, Q-1$.
Fix $0 < \theta < \zeta/2$. For $p= 1, \ldots, Q-1$ we define:
\alns{
A_n^{(p)}(\theta) = \Bigg[ \bigcup_{|\uv|=n} \bigg( \sum_{\uu \in \Iv^{(p)}} \delta_{b_n^\inv |\Xu|} \Big(\theta/n, \infty \Big) \ge 1 \bigg) \Bigg]^c.
}
We will show that
\aln{
\lim_{n \to \infty} \prob \bigg(\Big(A_n^\upp(\theta) \Big)^c \bigg) =0 \label{chap4_eq:negligible_lighter_tail}
}
for every $p=1, \ldots , Q-1$. Fix $p$. To prove \eqref{chap4_eq:negligible_lighter_tail} it suffices to show that
\alns{
\lim_{n \to \infty} \probt\bigg(\Big(A_n^\upp(\theta) \Big)^c \bigg) =0, \mbox{ almost surely}
}
where $\probt(\cdot)$ denotes the probability conditioned on underlying Galton-Watson tree.
Observe that
\aln{
\probt \bigg[ \Big( A_n^\upp(\theta) \Big)^c \bigg]  \le \sum_{|\uv|=n} \probt \bigg( \sum_{\uu \in \Iv^\upp} \delta_{n b_n^\inv |\Xu|} (\theta, \infty) \ge 1 \bigg) . \label{chap4_eq:upper_bound1_negligible_lighter_tail}
}
Note that on a path $\Iv$, the displacements corresponding to the $p^{th}$ type of particles are independent and they are also identically distributed. Thus, conditioned on $\bbt$,
$$ \sum_{\uu \in \Iv^\upp} \delta_{n b_n^\inv |\Xu|} (\theta, \infty) \sim \mbox{ Binomial} \bigg(|\Iv^\upp|, \prob \Big(n |X^\upp_1| > b_n \theta \Big) \bigg). $$
Using  $|\Iv^\upp| \le n$,  Assumption \ref{chap4_ass:displacement_multi} and taking a sequence $\{X^{(p,i)} : i \ge 1\}$ of i.i.d. copies of $X_1^\upp$, we get the following upper bound to the right hand side of \eqref{chap4_eq:upper_bound1_negligible_lighter_tail}:
\aln{
\sum_{|\uv| = n} \prob \bigg( \sum_{i=1}^n \delta_{n b_n ^\inv |X^{(p,i)}|} (\theta, \infty) \ge 1 \bigg) & = |D_n | \prob \bigg( \sum_{i=1}^n \delta_{n b_n ^\inv |X^{(p,i)}|} (\theta, \infty) \ge 1 \bigg) \nonumber \\
& \le (\mbox{const}) |D_n| n \prob(n |X_1^\upp| > b_n \theta) \nonumber \\
& \le (\mbox{const})  |D_n| n \varepsilon(n^\inv b_n \theta) \prob( n |X_1^\upQ| > b_n \theta), \label{chap4_eq:upper_bound2_negligible_lighter_tail}
}
where $\{X^{(p,i)} : i \ge 1\}$ is the collection of independent copies of $X_1^\upp$. Fix $\eta >0$.

Fix $\eta_1 > 0$ and $\eta_2 > 0$. It follows from \eqref{chap4_eq:marginal_bn} that  $b_n  = \rho^{n/\alpha} L'(\mu^{n/\alpha})$ where $L'(\cdot)$ is a slowly varying function. As $\{n^\inv b_n\}$ is an increasing sequence, we can use Potter's bound combined with upper bound in Lemma  \ref{chap4_lemma:upper_bound_varepsilon} to get
\aln{
\varepsilon(n^\inv b_n \theta) \le (\mbox{const.}) n^{\gamma + \eta_1} \rho^{-n(\gamma- \gamma')(\alpha^\inv - \beta)} \frac{\varepsilon(b_n \theta)}{\varepsilon(b_n)} \label{chap4_eq:potters_bound_varepsilon}
}
for large enough $n$. Again using Potter's bound, we get that
\aln{
\prob \Big(n |X_1^\upQ | > b_n \theta \Big) \le (1- \eta_2)^\inv n^{\alpha + \eta_2} \prob(|X_1^\upQ| > b_n \theta) \label{chap4_eq:potters_bound}
}
for sufficiently large $n$. Finally combining \eqref{chap4_eq:potters_bound} and \eqref{chap4_eq:potters_bound_varepsilon}, we can get the following upper bound for \eqref{chap4_eq:upper_bound2_negligible_lighter_tail} as
\aln{
\mbox{ const.} \times \frac{|D_n|}{\rho^n} \times \frac{\varepsilon(b_n \theta)}{\varepsilon(b_n)} \times \rho^n \prob \Big( |X_1^\upQ| > b_n \theta \Big) \times n^{1 + \gamma + \alpha + \eta_1 + \eta_2} \rho^{-n(\gamma- \gamma')(\alpha^\inv - \beta)} \label{chap4_eq:final_upper_bound_negligible_lighter_tail}
}
for $ \gamma' \in (0, \gamma)$ and $\beta \in (0, \alpha^\inv)$. Note that by \eqref{chap4_eq:Kesten_Stigum_limit_multi}
\aln{
\lim_{n \to \infty} \frac{|D_n|}{\rho^n} = \lim_{n \to \infty} \dfrac{\sum_{p=1}^Q |D_n^\upp|}{\rho^n} = \lim_{n \to \infty} \frac{|\sum_{p=1}^Q Z_n^\upp|}{\rho^n} = W
}
almost surely. The other factors in \eqref{chap4_eq:final_upper_bound_negligible_lighter_tail} converges to finite positive limit except the last factor which converges to $0$. This completes the proof of \eqref{chap4_eq:negligible_lighter_tail}.

To formalize the fact that on a path at most one displacement associated with the $Q^{th}$ type of particles can be large, we define
\aln{
A_n^{(Q)}(\theta)= \Bigg[ \bigcup_{|\uv|=n} \bigg( \sum_{\uu \in \Iv^{(Q)}} \delta_{b_n^\inv |\Xu|} \Big(\theta/n, \infty \Big) \ge 2 \bigg) \Bigg]^c.
}
Then adopting the method described in Subsection~4.4 in \cite{bhattacharya:hazra:roy:2014}, one can prove that
\aln{
\lim_{n \to \infty} \prob \bigg[ \Big( A_n^\upQ(\theta) \Big)^c \bigg]=0. \label{chap4_eq:Q_larrge_jump}
}

Now we are ready to prove \eqref{chap4_eq:goal_one_large_jump} using \eqref{chap4_eq:negligible_lighter_tail} and \eqref{chap4_eq:Q_larrge_jump}. We can bound the probability in \eqref{chap4_eq:goal_one_large_jump} by
\aln{
\prob \bigg( \sum_{|\uv=n|} \Big| f(b_n ^\inv \Sv) - \sum_{\uu \in \Iv^\upQ} f(b_n^\inv \Xu) \Big| > \epsilon, ~~~ \bigcap_{p=1}^Q A_n^\upp(\theta) \bigg) + o(1). \label{chap4_eq:upper_bound1_goal_one_large_jump}
}
Let $T_\uv^\upp$ denote the maxima in modulus among the displacements on the set $\Iv^\upp$ ($p=1, \ldots, Q$) where maxima in modulus  on an empty set is assumed to be zero. We will discuss now the possible scenarios appearing on the event considered in \eqref{chap4_eq:upper_bound1_goal_one_large_jump}. One possibility is that $\max_{1 \le p \le Q} \max_{|\uv| = n} n T_\uv^\upp < b_n \theta$ and it is clear that in this case its probability is zero. The another possibility is that $\max_{1 \le p \le Q-1} \max_{|\uv|=n} n T_\uv^\upp < b_n \theta $ and $\max_{|\uv| =n} n T_\uv^\upQ > b_n \theta$. Then the probability in \eqref{chap4_eq:upper_bound1_goal_one_large_jump} equals
\alns{
\prob \bigg( \sum_{|\uv|=n} \Big| f(b_n^\inv \Sv) - f(b_n^\inv T_\uv^\upQ) \Big| > \epsilon, \bigcap_{p=1}^Q A_n^\upp(\theta) \bigg).
}
Starting from this step, by molding the proof discussed in Subsection 4.4 in \cite{bhattacharya:hazra:roy:2014}, it is easy to conclude the assertion of the lemma.
\end{proof}

\subsection{Computation of the Weak Limit}\label{chap4_subsec:weak}

To compute the weak limit of the sequence of point processes $\tnkb$, we shall follow the method of ``regularization'' as used in \cite{bhattacharya:hazra:roy:2016}. Recall that we have decided to identify each edge with its vertex away from the root and assign its type to the edge. Also recall that $A_{\boldsymbol{\varpi}}^{(B)}$ denotes the number of descendants of the vertex $\boldsymbol{\varpi}$ if the $K^{th}$ generation of the pruned subtree containing $\boldsymbol{\varpi}$.  Keeping this abuse of notation in mind, we shall modify the pruned sub-trees $\bbt_{pj}(B)$ for $j= 1, \ldots, |D_{n-K}^\upp|$ and $p=1, \ldots, Q$ according to the following algorithm:

\been
\item[R1] Fix $p=1$ and $j=1$, i.e. look at $\bbt_{11}(B)$.

\item[R2] Look at the root of $\bbt_{11}(B)$. If it has exactly $B$ descendants of type $1$ in the next generation then keep it as it is. If it has $l$ ($< B$) descendants of type $1$, then add $(B-l) $ descendants (of type $1$) and define $\aub=0$, where $\uu$ is newly added edge. Next do the same for descendants of other types. Now the root has $QB$ descendants in total in the next generation. Now attach an independent copy of $(X^{(1)}_1, \ldots, X^{(1)}_B, \ldots, X^\upQ_1, \ldots, X^\upQ_B)$ to the descendants.

\item[R3] Repeat R2 for the particles at the first generation of $\bbt_{11}(B)$. Continue till the $K^{th}$ generation.

\item[R4] Repeat R2 and R3 for the other trees in $\bbf_1$.

\item[R5] Repeat R2, R3 and R4 for the forests $\bbf_p$ for $2 \le p \le Q$.

\een

We shall abuse the notation and  denote the regularized sub-trees also by $\bbt_{pj}(B)$. Let us concentrate on  $p=1$ and $j=1$. To each edge $\uu \in \bbt_{11}(B)$, we associate the triplet $(m,q,k)$ where $m$ denotes the generation of the edge $\uu$, $q$ denotes the type of the edge $\uu$ and $k$ denotes the enumeration of the vertex $\uu$ among all the vertices of type $q$ at the $m^{th}$ generation of $\bbt_{11}(B)$. Using this notation, we can write down the vector of displacements associated to the particles of type $Q$ in the tree $\bbt_{11}(B)$ as follows:
\alns{
\tilde{X}_1 = \Big( X'_{(1,Q,1)}, \ldots, X'_{(1,Q,B)}, X'_{(2,Q,1)}, \ldots, X'_{(2, Q, QB^2)}, \ldots, X'_{(K,Q,1)}, \ldots, X'_{(K, Q, Q^{K-1}B^K)} \Big).
}
It takes values in
\alns{
\tilde{\real}^B = \prod_{t=1}^{B+ QB^2 + \cdots + Q^{K-1}B^K} \real.
}
Similarly we can write,
\alns{
\tilde{A}_1 = (A^{(B)}_{(1,Q,1)}, \ldots, A^{(B)}_{(1,Q,B)}, A_{(2,B,1)}^{(B)}, \ldots, A_{(2,B,QB^2)}^{(B)}, \ldots, A^{(B)}_{(K,Q,1)}, \ldots, A^{(B)}_{(K,Q,Q^{K-1}B^K)})
}
taking values in
\alns{
\tilde{S}^B = \prod_{t=1}^{B + QB^2 + \ldots + Q^{K-1}B^K} \{0, \ldots, B^K \} .
}

Let $\mathscr{G}(\cdot)$ denote the law of $\tilde{A}_1$.
Recall that,
\alns{
\rho^n \prob \Big( b_n^\inv (X^\upQ_1, \ldots, X^\upQ_B) \in \cdot \Big) \hlconv \lambda^{(B)} (\cdot)
}
on the space $\real^B \setminus \{0_B\}$ where $0_B = (0,\ldots, 0) \in \real^B$ and $\lambda^{(B)} = \lambda \circ \mbox{PROJ}_B^\inv$ with $\mbox{PROJ}_B: \real^\bbn \to \real^B$ defined by $\mbox{PROJ}_B ((u_i)_{i=1}^\infty) = (u_1, \ldots, u_B)$. Using this  combined with the fact that the branching mechanism and the displacements are independent, we get that
\aln{
\rho^n \prob(b_n^\inv \tilde{X}_1 \in \cdot, \tilde{A}_1 \in \cdot ) \hlconv \tau \otimes \mathscr{G} \label{chap4_eq:disp_Au_joint_tree1}
}
on $\tilde{\real}^B \setminus \{\boldsymbol{0}_B\}$ where $\boldsymbol{0}_B \in \tilde{\real}^B$ with all its elements as $0$ and
\alns{
\tau &= \sum_{m=1}^K \sum_{l \in J_m} \tau_{ml}
}
with
\alns{
\tau_{ml} =  \bigotimes_{t=1}^{B+QB^2 + \ldots + Q^{m-2}B^{m-1} +l -1} \delta_0
\otimes \lambda^{(B)} \bigotimes_{t=B+QB^2 + \ldots + Q^{m-2}B^{m-1} +l +B}^{B+QB^2 + \ldots + Q^{K-1}B^K} \delta_0,
}
where $J_m = \{r \in \{1,2, \ldots, Q^{m-1}B^m\} : r \equiv 1 \mod B\}$.

It is clear that after the regularization, the point process associated to the tree $\bbt_{11}(B)$ is modified. But an important point to note, is that this modification does not affect the dependence structure prevailing among the displacements associated to a tree and also among the trees. In particular, it does not affect the joint distribution of the point processes associated to the trees and also to a specific tree. The point process of displacements (without the scaling) associated to the tree $\bbt_{p1}(B)$ will be denoted by $N^{'(K,B,p,1)}$:
\alns{
N^{'(K,B,p,1)} = \sum_{\uu \in \bbt_{p1}^\upQ(B)} \aub \delta_{\Xu}
}
for every $p=1, \ldots, Q$.

\begin{proof}[Proof of Lemma \ref{lemma:hlconv_t11_multi}]

Note that it suffices to prove that $N^{'(K,B,p,1)}$ is regularly varying in the space of all Radon point measures $\scrmrz$. In order to establish this fact, we have to show that
\alns{
m_n^\upp = \rho^n \prob \Big( \mbfs_{b_n^\inv} N^{'(K,B,p,j)} \in \cdot \Big) \hlconv m^*_p(\cdot)
}
for some $m^*_p \in \bbm(\scrmrz)$ described in \eqref{chap4_eq:limit_hl_mq}  for all $j=1, \ldots, |D_{n-K}^{(p)}|$. Following \cite{hult:samorodnitsky:2010}, it is clear that then we have to establish that for every $f_1, f_2 \in \ckr$ and $\epsilon_1, \epsilon_2 > 0$,
\aln{
&\int \bigg( 1- e^{ - \Big(\nu(f_1) - \epsilon_1 \Big)_+ } \bigg)\bigg( 1- e^{ - \Big(\nu(f_2) - \epsilon_2 \Big)_+ } \bigg) m_n^\upp(\dtv \nu) \nonumber \\
&\hspace{2cm} \to \int \bigg( 1- e^{ - \Big(\nu(f_1) - \epsilon_1 \Big)_+ } \bigg)\bigg( 1- e^{ - \Big(\nu(f_2) - \epsilon_2 \Big)_+ } \bigg) m^*_p(\dtv \nu)  \label{chap4_eq:defn_hls_functional}
}
as $n\to\infty$, where $\nu(f) = \int f(x) \nu(\dtvx)$. Now we shall consider the case, $p=1$ and the other cases can be dealt similarly.
Following the steps discussed in \cite{bhattacharya:hazra:roy:2016} and using \eqref{chap4_eq:disp_Au_joint_tree1}, we get that
\aln{
&\int \bigg( 1- e^{ - \Big(\nu(f_1) - \epsilon_1 \Big)_+ } \bigg)\bigg( 1- e^{ - \Big(\nu(f_2) - \epsilon_2 \Big)_+ } \bigg) m_n^{(1)}(\dtv \nu) \nonumber \\
&\to  \sum_{m=1}^K \int \sum_{\tilde{a} \in \tilde{S}^B} \sum_{l \in J'_m }  \mathscr{G}(\tilde{a}) \bigg( 1- \exp\bigg\{- \Big(  \sum_{j=l}^{l+B-1} a_{(m,Q,j)} f_1(x_{j-l+1}) -\epsilon_1  \Big)_+ \bigg\} \bigg) \nonumber \\
& \hspace{1.5cm} \bigg( 1- \exp\bigg\{- \Big(  \sum_{j=l}^{l+B-1} a_{(m,Q,j)} f_2(x_{j-l+1}) -\epsilon_2  \Big)_+ \bigg \} \bigg) \lambda^{(B)}(\dtv \boldsymbol{x})  \label{chap4_eq:lim_n_hls_m1}
}
with $J'_m = \{l \in J_m : (a_{(m,Q,l)}, \ldots, a_{(m,Q, l+B-1)}) \neq (0,\ldots, 0)\}$ for every $m =1, \ldots, K$. To
describe the above limit, we will introduce new variables as follows.

\begin{itemize}

\item Let $\{U_{1i}^\upp(Q,B) : p=1,2, \ldots, Q \mbox{ and } i \ge 1\}$ be a collection of independent random variables such that for every $p \in \{1,2, \ldots,Q\}$, $U_{1i}^\upp(Q,B) \eqd Z_1^\upp(Q,B)$ for all $i \ge 1$.

\item We denote by $(Z^{(p)}_{i} (1,B), \ldots, Z^{(p)}_{i}(Q,B))$ the random vector of the number of particles at the $i^{th}$ generation for different types when root is of type $p$
in the pruned tree.
Let $\{\Delta_{K-m,i}^\upQ (B) : m=1,2, \ldots, K, i \ge 1\}$ be a collection of independent random variables such that for every fixed $m \in \{1,2, \ldots, K\}$,
\alns{
\Delta^\upQ_{K-m, i} (B) \eqd \sum_{q=1}^Q Z_{K-m}^\upQ (q,B)
}
for all $i \ge 1$. Assume further that the collection $\{\Delta_{K-m,i}^\upQ (B) : m=1, \ldots, K, i \ge 1\}$ is independent of the collection $\{U_{1i}^\upp(Q,B) : p=1, \ldots, Q \mbox{ and } i \ge 1\}$.
\end{itemize}

Fix a generation $1 \le m \le K$.
In order to compute the expectation under the law $\mathscr{G}(\cdot)$, we have to count only those vertices at the $m^{th}$ generation of type $Q$ with different parents in the $(m-1)^{th}$ generation which have at least one descendant at the $K^{th}$ generation of $\bbt_{11}(B)$.
First observe that, only the newly added vertex at the $(m-1)^{th}$ generation has children at the $m^{th}$ generation such that each of the children has zero descendant at the $K^{th}$ generation. Thus we start with the total number of particles at the $(m-1)^{th}$ generation (who are potential parents of the particles of type $Q$ at the $m^{th}$ generation) which is distributed according as $\sum_{q=1}^Q Z^{(1)}_{m-1} (q,B)$ (recall that root of $\bbt_{11}(B)$ is of type $1$). Secondly observe that, the number of particles at the $(m-1)^{th}$ generation is independent of the random vector denoting the number of children of them at the $m^{th}$ generation of type $Q$. It is clear that a particle of type $p'$ at the $(m-1)^{th}$ generation produces a random number of  children of type $Q$ according to the law $Z^{(p')}_1(Q,B)$. The third observation is that the number of descendants of a particle of type $Q$ at the $m^{th}$ generation is independent of the number of descendants of the other particles of type $Q$ at the $m^{th}$ generation. It is also
independent of the number of particles of type $Q$ at the $m^{th}$ generation and of the total number of particles at the $(m-1)^{th}$ generation.  Moreover, it is distributed according to the law of
\alns{
\sum_{q=1}^Q Z_{K-m}^{(Q)}(q,B).
}
The number of children of type $Q$ at the $m^{th}$ generation of the $i^{th}$ particle of type $p'$ at the $(m-1)^{th}$ generation has then the law of $U_{1i}^{(p')} (Q,B)$ for every $i =1, \ldots, Z_{m-1}^{(1)}(p,B)$ and $p=1, \ldots, Q$ which is independent of the vector  $(Z^{(1)}_{m-1} (1,B), \ldots, Z^{(1)}_{m-1}(Q,B))$. Moreover, the number of descendants at the $K^{th}$  generation of $\bbt_{11}(B)$ of the $i^{th}$ particle of type $Q$ at the $m^{th}$ generation has the law of $\Delta_{K-m,i}^\upQ(B)$ for every $i \ge 1$ which is also independent of $(Z^{(1)}_{m-1} (1,B), \ldots, Z^{(1)}_{m-1}(Q,B))$. Recall that $\bfe_i$ denotes the $i^{th}$ unit vector which has all components $0$ except the $i^{th}$ one for all $i=1,2, \ldots,Q.$ Hence the limiting expression in \eqref{chap4_eq:lim_n_hls_m1} becomes:
\aln{
& \sum_{m=1}^K \int \exptn \bigg[ \sum_{q=1}^Q \sum_{l=1}^{Z_{m-1}^{(1)}(q,B)}\bigg( 1- \exp \bigg\{ - \Big( \sum_{j=1}^{U_1^\upq(Q,B)} \Delta_{K-m,j}^\upQ(B) f_1(x_j) - \epsilon_1 \Big)_+ \bigg\}  \bigg) \nonumber \\
&\hspace{2cm} \bigg( 1- \exp \bigg\{ - \Big( \sum_{j=1}^{U_1^\upq(Q,B)} \Delta_{K-m,j}^\upQ(B) f_2(x_j) - \epsilon_2 \Big)_+ \bigg\}  \bigg) \bigg] \lambda(\dtv(x_1, \ldots, x_B)) \nonumber \\
&= \sum_{m=1}^K  \sum_{q=1}^Q \bfe_1^t \Big( M(B) \Big)^{m-1} \bfe_q \int \exptn \bigg[\bigg( 1- \exp \bigg\{ - \Big( \sum_{j=1}^{U_1^\upq(Q,B)} \Delta_{K-m,j}^\upQ(B) f_1(x_j) - \epsilon_1 \Big)_+ \bigg\}  \bigg) \nonumber \\
& \hspace{1cm} \bigg( 1- \exp \bigg\{ - \Big( \sum_{j=1}^{U_1^\upq(Q,B)} \Delta_{K-m,j}^\upQ(B) f_2(x_j) - \epsilon_2 \Big)_+ \bigg\}  \bigg) \bigg] \lambda(\dtv(x_1, \ldots, x_B)). \label{chap4_eq:exptn_exponent_m1}
}
In the last step, we are  using the Wald's identity. Note now that combing \eqref{chap4_eq:exptn_exponent_m1} and  \eqref{chap4_eq:lim_n_hls_m1} we can write down the limit in the form of right hand side of \eqref{chap4_eq:defn_hls_functional} (with $p=1$) where
\alns{
m_1^*(\cdot)  = \sum_{m=1}^K \sum_{q=1}^Q \bfe_1^t \Big(M(B) \Big)^{m-1} \bfe_q \exptn \bigg[ \lambda^{(B)} \bigg\{ (x_1, \ldots, x_B) \in \real^B : \sum_{j=1}^{U_1^\upq(Q,B)} \Delta_{K-m,j}^\upQ (B) \delta_{x_j} \in \cdot \bigg\} \bigg].
}
Similarly, for every $p=2, \ldots, Q$,
\aln{
m^*_p(\cdot) = \sum_{m=1}^K \sum_{q=1}^Q \bfe_p^t M^{m-1}(B) \bfe_q \exptn \bigg[ \lambda^{(B)} \bigg\{ (x_1, \ldots, x_B) \in \real^B : \sum_{j=1}^{U_1^\upq(Q,B)} \Delta_{K-m,j}^\upQ (B) \delta_{x_j} \in \cdot \bigg\} \bigg]. \label{chap4_eq:limit_hl_mq}
}
Following the same steps discussed in \cite{bhattacharya:hazra:roy:2016}, it can be verified that $m_p^* \in \bbm(\scrmrz)$ for every $p=1, \ldots, Q$.
\end{proof}

We are ready now to prove the main lemma of this section.

\begin{proof}[Proof of Lemma \ref{lemma:weaklimit_multi}] 
Recall that in the regularization step, we replaced the displacements corresponding the the particles of type $Q$ coming from same parent by an independent copy of $(X_1^\upQ, \ldots, X_B^\upQ)$. The new displacements corresponding to the vertex $\boldsymbol{\varpi}$ is denoted by $X'_{\boldsymbol{\varpi}}$.
By regularization construction, to prove the first statement it suffices to show the weak convergence
$\dnkb \Rightarrow N_*^{(K,B)}$ as $n \to \infty$
for
\alns{
\dnkb = \sum_{p=1}^Q \sum_{j=1}^{|D_{n-K} ^\upp| } \sum_{\uu \in \bbt_{pj}^{\upQ}(B)} \aub \delta_{b_n^\inv \Xu^\prime}.
}
Now we will consider the Laplace functional of $\dnkb$.

Consider a function $f \in \ckr$. We then have to find the limit of
\aln{
&\exptn \bigg[ \exp \bigg\{ - \dnkb(f) \bigg\} \bigg] \nonumber\\
&= \exptn \bigg[ \exp \bigg\{ - \sum_{p=1}^Q \sum_{j=1}^{Z_{n-K}(p)} \sum_{\uu \in \tpjq(B)} \aub f({b_n^\inv X'_\uu}) \bigg\} \bigg] \nonumber \\
&= \prod_{p=1}^Q \exptn \bigg[ \exptn \bigg( \exp \bigg\{ - \sum_{j=1}^{Z_{n-K}(p)} \sum_{\uu \in\tpjq(B)} \aub f({b_n^\inv X'_\uu}) \bigg\} \bigg| \calf_{n-K} \bigg) \bigg]. \label{chap4_eq:laplace_dnkb}
}
First we will focus on above conditional expectation and we will use the fact that $\{\aub : \uu \in \tpjq(B), j \ge 1 \}$ and $\{\Xu' : \uu \in \tpjq(B), j \ge 1 \}$ are independent of $\calf_{n-K}$. Note that it equals:
\aln{
\bigg[\exptn \bigg( \exp \bigg\{ - \sum_{\uu \in \bbt_{p1}^\upQ (B)} \aub f(b_n^\inv X'_\uu) \bigg\} \bigg) \bigg]^{Z_{n-K}(p)} \label{chap4_eq:lpalace_of_conditioned_dnkb}
}
since the point processes associated to the trees $\tpjq(B)$ are i.i.d.~for every fixed $p \in \{1, \ldots,Q\}$. Using fact that
\alns{
\bigg[\exptn \bigg(  \exp\bigg\{ - \mbfs_{b_n^\inv} N'^{(K,B,p,1)}(f) \bigg\} \bigg) \bigg]^{\rho^n} \to \exp \bigg\{-\int \Big(1-e^{-\nu(f)}\Big) m_1^*(\dtv \nu) \bigg\}
}
we can conclude that
\alns{
\bigg[ \exptn \bigg( \exp \bigg\{ - \sum_{\uu \in \bbt_{p1}^\upQ(B)} \aub f(b_n^\inv \Xu') \bigg\} \bigg) \bigg]^{\rho^n}
}
converges to
\alns{
&\exp \bigg\{ - \sum_{m=1}^K \sum_{q=1}^Q \bfe_p^t M^{m-1}(B) \bfe_q  \\
&\;\;\;\;\;\;\;\;\;\;\;\;\;\;\;\;\;\;\;\;\;\;\;\;\;\;\int \exptn \bigg( 1 - \exp \bigg\{ - \sum_{j=1}^{U_1^\upq(Q,B)} \Delta^\upQ_{K-m,j}(B) f(x_j) \bigg\} \bigg) \lambda^{(B)}(\dtv(x_1, \ldots, x_B)) \bigg\}
}
as $n \to \infty$. From the Kesten-Stigum Theorem we get that \eqref{chap4_eq:lpalace_of_conditioned_dnkb} converges to
\alns{
&\exp \bigg\{ - \frac{1}{\rho^K} W \varsigma_p \sum_{m=1}^K \sum_{q=1}^Q \bfe_p^t (M(B))^{m-1} \bfe_q \\
& \hspace{2cm} \int \exptn \bigg( 1 - \exp \bigg\{ - \sum_{j=1}^{U_1^\upq(Q,B)} \Delta^\upQ_{K-m,j}(B) f(x_j) \bigg\} \bigg) \lambda^{(B)}(\dtv(x_1, \ldots, x_B))\bigg\}
}
almost surely as $n \to \infty$. Finally, using dominated convergence theorem, we get that the convergence of the Laplace functional \eqref{chap4_eq:laplace_dnkb} to:
\aln{
&\exptn \bigg[ \exp\bigg\{-  \frac{1}{\rho^K}  W \sum_{p=1}^Q \varsigma_p \sum_{m=1}^K \sum_{q=1}^Q \bfe_p^t (M(B))^{m-1} \bfe_q  \nonumber \\
&\hspace{1.5cm} \int \exptn \bigg( 1 - \exp \bigg\{ - \sum_{j=1}^{U_1^\upq(Q,B)} \Delta^\upQ_{K-m,j}(B) f(x_j) \bigg\} \bigg) \lambda^{(B)}(\dtv(x_1, \ldots, x_B))\bigg\}  \bigg]. \label{chap4_eq:limitn_laplace_functional_dnkb}
}
Here we shall construct a point process which has the same Laplace functional as obtained in \eqref{chap4_eq:limitn_laplace_functional_dnkb}. Consider a Poisson random measure
\alns{
\poi^{(B)} = \sum_{l=1}^\infty \delta_{\pmb{\varrho}_l} = \sum_{l=1}^\infty \delta_{(\varrho_{l1}, \ldots, \varrho_{lB})}
}
on $\real^B$ with intensity measure $\lambda^{(B)}(\cdot)$ and independent of the random variable $W$. Let $(G^{(B)}, \boldsymbol{T}^{(B)})$ be a $\Xi$-valued random variable where $\Xi = \cup_{i \in \bbn} \{i\} \times \bbn^i$ with probability mass function
\alns{
& \prob \Big(G^{(B)} = g, \boldsymbol{T}^{(B)} = (t_1, \ldots, t_g) \Big) \\
&= \frac{1}{s_B \rho^K} \sum_{p=1}^Q \varsigma_p \sum_{m=1}^K \sum_{q=1}^Q \bfe_p (M(B))^{m-1} \bfe_q \prob( U_1^\upq (Q,B) = g) \prod_{j=1}^g \prob(\Delta_{K-m, j}^\upQ = t_j)
}
for all  $(t_1, \ldots, t_g) \in \bbn^g \mbox { and } g \in \bbn$, where $s_B$ is the normalizing constant that makes the above a probability measure on $\Xi$. Now consider a collection $\{(G_l^{(B)}, \boldsymbol{T}_l^{(B)}): l \ge 1\}$ of independent copies of $(G^{(B)}, \boldsymbol{T}^{(B)})$ and also independent of $W$ and $\poi^{(B)}$. Now the Cox cluster process
\alns{
N_*^{(K,B)} = \sum_{l=1}^\infty \sum_{k=1}^{G_l^{(B)}} T_{lk} \delta_{(s_B W)^{1/\alpha}\varrho_{lk}}
}
can be shown to have the same Laplace functional as in \eqref{chap4_eq:limitn_laplace_functional_dnkb} following the method discussed in \cite{bhattacharya:hazra:roy:2016}.

Now, we let $B \to \infty$ and using dominated convergence theorem once again, we get from the right hand side of \eqref{chap4_eq:limitn_laplace_functional_dnkb},
\aln{
&\exptn \bigg[ \exp\bigg\{-  \frac{1}{\rho^K}  W \sum_{p=1}^Q \varsigma_p \sum_{m=1}^K \sum_{q=1}^Q \bfe_p^t M^{m-1} \bfe_q \nonumber \\
&\hspace{2cm} \int \exptn \bigg( 1 - \exp \bigg\{ - \sum_{j=1}^{U_1^\upq(Q)} \Delta^\upQ_{K-m,j} f(x_j) \bigg\} \bigg) \lambda(\dtv\boldsymbol{x})\bigg\}  \bigg] \label{eq:Laplcae_fun_mult_fin}
}
Consider a $\Xi$-valued random variable $(G', \boldsymbol{T}')$ with the probability mass function
\aln{
&\prob(G'=g, \boldsymbol{T}'=(t_1, \ldots, t_g)) \nonumber \\
&= \frac{1}{s' \rho^K} \sum_{p=1}^Q \varsigma_p \sum_{m=0}^{K-1} \sum_{q=1}^Q \bfe_p^t M^{K-m-1} \bfe_q \prob(U_1^\upq (Q) =g) \prod_{j=1}^g \prob(\Delta^\upQ_{m,j} = t_j) \label{chap4_eq:limitB_Laplace_dnkb}
}
for every $(t_1, \ldots, t_g) \in \bbn^g$ and $g \in \bbn$, where $s'$ is the normalizing constant. Now consider a collection $\{(G'_l, \boldsymbol{T}'_l) : l \ge 1 \}$ of independent copies of $(G', \boldsymbol{T}')$ which is also independent of $\poi$ and $W$. Now the following point process
\alns{
N_*^{(K)} = \sum_{l=1}^\infty \sum_{k=1}^{G'_l} T'_{lk} \delta_{(s' W)^{1/\alpha} \xi_{lk}}
}
which has the same Laplace functional as described in \eqref{eq:Laplcae_fun_mult_fin}. It is easy to see that the right hand side of \eqref{chap4_eq:limitB_Laplace_dnkb}, becomes
\alns
{
&  \exptn \bigg[ \exp\bigg\{-  \frac{1}{\rho^K}  W \sum_{p=1}^Q \varsigma_p \sum_{m=0}^{K-1} \sum_{q=1}^Q \bfe_p^t M^{K-m-1} \bfe_q \\
&\hspace{2cm} \int \exptn \bigg( 1 - \exp \bigg\{ - \sum_{j=1}^{U_1^\upq(Q)} \Delta^\upQ_{m,j} f(x_j) \bigg\} \bigg) \lambda(\dtv\boldsymbol{x})\bigg\}  \bigg]\\
&= \exptn \bigg[ \exp\bigg\{-  \frac{1}{\rho^K}  W \sum_{p=1}^Q \varsigma_p \sum_{m=0}^{K-1} \sum_{q=1}^Q \bfe_p^t (\rho^{K-m-1} P + R^{K-m-1}) \bfe_q \\
&\hspace{2cm} \int \exptn \bigg( 1 - \exp \bigg\{ - \sum_{j=1}^{U_1^\upq(Q)} \Delta^\upQ_{m,j} f(x_j) \bigg\} \bigg) \lambda(\dtv\boldsymbol{x})\bigg\}  \bigg]
}
using rearrangement of terms and Theorem~1  of Chapter V (page~185) in \cite{athreya:ney:1972}. Here $R$ is as in Assumption \ref{chap4_ass:ass_branching_multi}

 Note that Theorem~1  of Chapter V (page~185) in \cite{athreya:ney:1972}  also implies
\alns{
& \frac{1}{\rho^K}  \sum_{p=1}^Q \varsigma_p \sum_{m=0}^{K-1} \sum_{q=1}^Q  \Big|\bfe_p^t  R^{K-m-1} \bfe_q \Big|
 \int \exptn \bigg( 1 - \exp \bigg\{ - \sum_{j=1}^{U_1^\upq(Q)} \Delta^\upQ_{m,j} f(x_j) \bigg\} \bigg)  \lambda(\dtv\boldsymbol{x})\\
& \le \Big( \frac{\rho_0}{\rho}  \Big)^K cQ \sum_{p=1}^Q \varsigma_p \sum_{m=0}^{K-1} \frac{1}{\rho_0^{m+1}}
}
for some $1 < \rho_0 < \rho$, which is negligible as $K \to \infty$ . Hence we can let $K \to \infty$ and the Laplace functional then becomes
\aln{
& \exptn \bigg[ \exp\bigg\{-   W \sum_{p=1}^Q \varsigma_p \sum_{m=0}^{\infty} \frac{1}{\rho^{m+1}} \sum_{q=1}^Q \bfe_p^t P \bfe_q \nonumber \\
&\hspace{2cm} \int \exptn \bigg( 1 - \exp \bigg\{ - \sum_{j=1}^{U_1^\upq(Q)} \Delta^\upQ_{m,j} f(x_j) \bigg\} \bigg) \lambda(\dtv\boldsymbol{x})\bigg\}  \bigg] \nonumber \\
& = \exptn \bigg[ \exp\bigg\{-   W \sum_{p=1}^Q \varsigma_p \sum_{m=0}^{\infty} \frac{1}{\rho^{m+1}} \sum_{q=1}^Q \bfe_p^t \vartheta \varsigma^t \bfe_q \nonumber \\
&\hspace{2cm} \int \exptn \bigg( 1 - \exp \bigg\{ - \sum_{j=1}^{U_1^\upq(Q)} \Delta^\upQ_{m,j} f(x_j) \bigg\} \bigg) \lambda(\dtv\boldsymbol{x})\bigg\}  \bigg] \nonumber \\
& =\exptn \bigg[ \exp\bigg\{-   W \sum_{p=1}^Q \varsigma_p \vartheta_p \sum_{m=0}^{\infty} \frac{1}{\rho^{m+1}} \sum_{q=1}^Q u_q \nonumber \\
&\hspace{2cm} \int \exptn \bigg( 1 - \exp \bigg\{ - \sum_{j=1}^{U_1^\upq(Q)} \Delta^\upQ_{m,j} f(x_j) \bigg\} \bigg) \lambda(\dtv\boldsymbol{x})\bigg\}  \bigg] \nonumber \\
& =\exptn \bigg[ \exp\bigg\{-   W  \sum_{m=0}^{\infty} \frac{1}{\rho^{m+1}} \sum_{q=1}^Q \varsigma_q \nonumber \\
&\hspace{2cm} \int \exptn \bigg( 1 - \exp \bigg\{ - \sum_{j=1}^{U_1^\upq(Q)} \Delta^\upQ_{m,j} f(x_j) \bigg\} \bigg) \lambda(\dtv\boldsymbol{x})\bigg\}  \bigg],  \label{chap4_eq:limit_laplacetransform}
}
where we used the fact that $P = \vartheta. \varsigma^t$ for $\vartheta^t. \varsigma =1$.

The last step will be to show that the expression obtained in \eqref{chap4_eq:limit_laplacetransform} is the expression for the Laplace functional of the point process $N_*$ as described in Theorem \ref{thm:main_thm_multi}.
To compute the Laplace functional of $N_*$, we will use auxiliary marked Cox process:
\alns{
P_* = \sum_{l=1}^\infty \delta_{\Big(G_l, T_{l1}, \ldots, T_{l,G_l}, ((\rho-1)^\inv W)^{1/\alpha} \pmb{\xi}_l \Big)}
}
with the state space as $N \times \real^\bbn$ for $N = \cup_{i=1}^\infty (\{i\} \times \bbn^i)$. We consider a positive, bounded and  continuous function $f$ on the metric space $N \times \real^\bbn$ which vanishes outside a neighbourhood of $N \times \{\boldsymbol{0}\}$. Then we get
\alns{
&\exptn \bigg(\exp \bigg\{ - P_*(f) \bigg\} \bigg) \nonumber \\
&= \exptn \bigg[  \exptn \bigg(\exp \bigg\{ - \sum_{l=1}^\infty f\Big(G_l, T_{l1}, \ldots, T_{lG_l}, ((\rho-1)^\inv W)^{1/\alpha} \pmb{\xi}_l \Big) \bigg\} \bigg| W \bigg) \bigg]. \nonumber
}
Now using the fact that $P_*$ conditioned on $W$ is a marked Poisson point process with independent marks and following Proposition 3.8 of \cite{resnick:1987},
we can conclude that the right hand side of above equality becomes
\alns{
\exptn \big[ \exp \bigg\{ - \int_{\rnz}  \exptn \bigg( 1- \exp \Big\{ - f\Big(G, T_1, \ldots, T_G, ((\rho-1)^\inv W)^{-1/\alpha} \boldsymbol{x}\Big) \Big\} \bigg| W \bigg) \lambda(\dtv \boldsymbol{x})\bigg\} \bigg].
}
Recall from \eqref{chap4_eq:hlconv_bn} that $\lambda(a. \boldsymbol{x}) = a^{-\alpha} \lambda(a)$ for every $a > 0$. Hence above expression equals
\alns{
\exptn \big[ \exp \bigg\{ - (\rho-1)^\inv W\int_{\rnz}  \exptn \bigg( 1- \exp \Big\{ - f \Big(G, T_1, \ldots, T_G,  \boldsymbol{x}\Big) \Big\} \bigg) \lambda(\dtv \boldsymbol{x})\bigg\} \bigg].
}

Let $f': \real \to \real_+$ be a bounded, continuous function vanishing on a neighbourhood of $0$. Then
\alns{
f(g, t_1, \ldots, t_g, \boldsymbol{x}) = \sum_{m=1}^g t_m f'(x_m)
}
is a positive, bounded and continuous function on $N \times \real^\bbn$ vanishing in a neighbourhood of $N \times \{\boldsymbol{0}\}$. Finally, we obtain
\alns{
& \exptn \bigg( \exp \bigg\{ - P_*(f) \bigg\} \bigg) \nonumber \\
&= \exptn \bigg( \exp \bigg\{ - N_*(f') \bigg\} \bigg) \nonumber \\
& = \exptn \bigg[ \exp \bigg\{ - (\rho-1)^\inv W \int_{\rnz} \exptn \bigg( 1- \exp\Big\{  - \sum_{j=1}^G T_j f'(x_j)\Big\} \bigg) \lambda(\dtv \boldsymbol{x}) \bigg\} \bigg] \nonumber \\
& = \exptn \bigg[ \exp \bigg\{ - (\rho-1)^\inv W \int_{\rnz}  \sum_{q=1}^Q \varsigma_q \sum_{g=1}^ \infty \prob(U_1^\upq(Q) =g)   \nonumber \\
&\;\;\;\;\;\;\;\;\;\;\;\;\;\;\;\;\;\;\;\;\;\;\;\;\;\;\;\;\;\;\;\;\;\;\;\;\;\;\;\;\;\;\;\;\;\;\;\;\;\;\;\;\;\;\;\;\;\;\;\;\;\;\;\;\;\;\; \exptn \bigg( 1- \exp\Big\{  - \sum_{j=1}^g T_j f'(x_j)\Big\} \bigg)\lambda(\dtv\boldsymbol{x}) \bigg\} \bigg] \nonumber \\
& = \exptn \bigg[ \exp \bigg\{ - (\rho-1)^\inv W \int_{\rnz}  \sum_{q=1}^Q \varsigma_q \sum_{g=1}^ \infty \prob(U_1^\upq(Q) =g)  (\rho-1) \nonumber \\
& \hspace{2cm}  \sum_{m=0}^ \infty \dfrac{1}{\rho^{m+1}} \sum_{t_1, \ldots, t_g \in \bbn^g} \prod_{j=1}^g \prob(\Delta_{m,j}^\upQ =t_j)   \bigg( 1- \exp\Big\{  - \sum_{j=1}^g t_j f'(x_j)\Big\} \bigg) \lambda(\dtv \boldsymbol{x}) \bigg\} \bigg] \nonumber \\
&= \exptn \bigg[ \exp \bigg\{ - W   \sum_{q=1}^Q \varsigma_q    \sum_{m=0}^ \infty \dfrac{1}{\rho^{m+1}} \int_{\rnz} \exptn   \bigg( 1- \exp\Big\{  - \sum_{j=1}^{U_1^\upq(Q)} \Delta_{m,j}^\upQ f'(x_j)\Big\} \bigg) \lambda(\dtv \boldsymbol{x}) \bigg\} \bigg].
}
which is same as that obtained in \eqref{chap4_eq:limit_laplacetransform} and thus $N_*$ is the limit point process.
\end{proof}

\noindent{\bf{Acknowledgements.}} This research was partially supported by the project RARE-318984 (a Marie Curie FP7 IRSES
Fellowship). Zbigniew Palmowski was partially supported by National Science Centre Grant under the grant 2015/17/B/ST1/01102. Ayan Bhattacharya acknowledges the warm hospitality of the Mathematical Institute, University of Wroc\l aw for his visit during the period September 10 - October 19, 2015. Authors are also thankful to Rajat Subhra Hazra for reading the first draft of the paper carefully and giving valuable suggestions.\\

\bibliographystyle{plain}	
\bibliography{ayanbib}

\newpage

\noindent Ayan Bhattacharya ({\tt ayanbhattacharya.isi@gmail.com})\\
Krishanu Maulik ({\tt krishanu@isical.ac.in})\\
Parthanil Roy ({\tt parthanil.roy@gmail.com})\\
Statistics and Mathematics Unit\\
Indian Statistical Institute\\
203 B. T. Road\\
Kolkata 700108, India\\

\noindent Zbigniew Palmowski ({\tt zbigniew.palmowski@gmail.com})\\
Faculty of Pure and Applied Mathematics\\
Wroc\l aw University of Science and Technology\\
ul. Wyb. Wyspia\'nskiego 27\\
50-370 Wroc\l aw, Poland\\

\end{document}